\documentclass[12pt]{amsart}
\usepackage{amssymb}
\usepackage{amsfonts}
\usepackage{graphicx}
\usepackage{amscd}

\newtheorem{theorem}{Theorem}
\theoremstyle{plain}
\newtheorem*{acknowledgement}{Acknowledgement}

\newtheorem{corollary}{Corollary}

\newtheorem{lemma}{Lemma}

\newtheorem*{remark}{Remarks}

\numberwithin{equation}{section}

\input{tcilatex}

\begin{document}
\title[Common attractive points]{Iterative approximation of common
attractive points of further generalized hybrid mappings}
\author{Safeer Hussain Khan}
\address{Safeer Hussain Khan, Department of Mathematics,Statistics and
Physics, Qatar University, Doha 2713, State of Qatar.}
\email{safeerhussain5@yahoo.com; safeer@qu.edu.qa}

\begin{abstract}
Our purpose in this paper is (i) to introduce the concept of further
generalized hybrid mappings (ii) to introduce the concept of common
attractive points (CAP) (iii) to write and use Picard-Mann iterative process
for two mappings. We approximate common attractive points of further
generalized hybrid mappings by using iterative process due to Khan \cite{SHK}
generalized to the case of two mappings in Hilbert spaces without closedness
assumption. Our results are generalizations and improvements of several
results in the literature in different ways.
\end{abstract}

\maketitle

\section{Introduction and Preliminaries}

Let $\mathbb{N}$ denote the set of positive integers and $\mathbb{R}$ the
set of real numbers. Let $H$ be a real Hilbert space and $C$ a nonempty
subset of $H$. Let $T$ be a mapping of $C$ into $H$. Recall that the set of
fixed points of $T$ is denoted and defined by $F(T)=\{z\in C:Tz=z\}$.
Takahashi and Takeuchi \cite{TT11} introduced the concept of attractive
points in Hilbert spaces. They defined and denoted the set of attractive
points as follows.%
\begin{equation*}
A(T)=\{z\in H:\left\Vert Tx-z\right\Vert \leq \left\Vert x-z\right\Vert \}\
\ \ \text{for all }x\in C.
\end{equation*}

From this definition, neither an attractive point is a fixed point nor
conversely. However, for a relation between the two, see Lemmas $\ref{ATFT1}$
and $\ref{ATFT2}$ below. Basically this concept was introduced to get rid of
hypothesis of closedness and convexity as used in a well-celebrated
Baillon's nonlinear ergodic theorem in Hilbert spaces \cite{B75}. They also
proved an existence theorem for attractive points without convexity in
Hilbert spaces. In these theorems, they used the so-called generalized
hybrid mappings (to be defined in the sequel) whose class is larger than the
class of nonexpansive mappings used in Baillon's theorem. Since we are
interested in existence theorem, we state it as follows.

\begin{theorem}
\label{Th1}(Takahashi and Takeuchi \cite{TT11}) Let $H$ be a Hilbert space
and $C$ a nonempty subset of $H$. Let $T:C\rightarrow C$ be a generalized
hybrid mapping Then $T$ has an attractive point if and only if $\exists
~z\in C$ such that $\{T^{n}z:n=0,1,\cdots \}$ is bounded.
\end{theorem}

Obviously, the hypothesis does not require any closedness or convexity.
Takahashi and Takeuchi \cite{TT11} also gave some properties of the
attractive points as follows.

\begin{lemma}
\label{ATFT1} Let $H$ be a real Hilbert space and let $C$ be a nonempty
closed convex subset of $H$. Let $T:C\rightarrow C$ be a mapping. If $%
A(T)\neq \varnothing $ then $F(T)\neq \varnothing .$
\end{lemma}

\begin{lemma}
\label{CC} Let $H$ be a real Hilbert space and let $C$ be a nonempty subset
of $H$. Let $T:C\rightarrow H$ be a mapping. Then $A(T)$ is a closed and
convex subset of $H$.
\end{lemma}

Later, the following was noted by Takahashi et al. \cite{TWY}\ for
quasi-non-expansive mappings.

\begin{lemma}
\label{ATFT2} Let $H$ be a real Hilbert space and let $C$ be a nonempty
subset of $H$. Let $T:C\rightarrow H$ be a quasi-nonexpansive mapping $($%
that is, $\left\Vert Tx-z\right\Vert \leq \left\Vert x-z\right\Vert ,$ $z\in
F(T))$. Then $A(T)\cap C=F(T).$
\end{lemma}

Let $l^{\infty }$ be the Banach space of bounded sequences with supremum
norm and $(l^{\infty })^{\ast }$ be its dual space (set of all continuous
linear functionals on $l^{\infty }$). It is well-known that there exists a $%
\mu \in (l^{\infty })^{\ast }$ (that is, there exists a continuous linear
functional on $l^{\infty }$) such that $\left\Vert \mu \right\Vert =\mu
(1)=1 $ and $\mu _{n}(x_{n+1})=\mu _{n}(x_{n})$ for each $%
x=(x_{1},x_{2},x_{3},\cdots )\in l^{\infty }.$ Such a $\mu $ is called a
Banach limit. Sometimes $\mu _{n}(x_{n})$ is denoted by $\mu (x).$ It is
also known that for a Banach limit $\mu ,\liminf_{n\rightarrow \infty
}x_{n}\leq \mu (x)\leq \limsup_{n\rightarrow \infty }x_{n}$ for each $%
x=(x_{1},x_{2},x_{3},\cdots )\in l^{\infty }.$\ As a special case, if $%
\lim_{n\rightarrow \infty }x_{n}$ exists and is $a,$ then $\mu (x)=a$ too.
This means the idea of a Banach limit is an extension of the idea of ususal
limits. It is also a well-known result that for a bounded sequence $\left\{
x_{n}\right\} $ in a Hilbert space $H$, there exists a unique $u_{0}\in 
\overline{co}\{x_{n}:n\in \mathbb{N}\}$ such that $\mu _{n}\left\langle
x_{n},v\right\rangle =\left\langle u_{0},v\right\rangle $ for all $v\in H.$

Recall that for every closed convex subset $C$ of a Hilbert space $H$, there
exists the metric projection $P_{C}:H\rightarrow C$. That is, for each $x\in
H,$ there is a unique element $P_{C}x\in C$ such that $\left\Vert
x-P_{C}x\right\Vert \leq \left\Vert x-y\right\Vert $ for all $y\in C.$ We
also need the following lemma due to Takahashi and Toyoda \cite{TT}.

\begin{lemma}
\label{TT} Let $K$ be a nonempty closed convex subset of a real Hilbert
space $H.$ Let $P_{K}:H\rightarrow K$ be the metric projection. Let $%
\{x_{n}\}$ be a sequence in $H.$ If $\left\Vert x_{n+1}-k\right\Vert \leq
\left\Vert x_{n}-k\right\Vert $ for any $k\in K$ and $n\in \mathbb{N},$ then 
$\{P_{K}x_{n}\}$ converges strongly to some $k_{0}\in K.$
\end{lemma}

Mathematicians started working on attractive points in various directions
after the publication of Theorem \ref{Th1}, see for example, \cite{AT}, \cite%
{CGZ}, \cite{GT}, \cite{KIK}, \cite{KwT}, \cite{KTY}, \cite{KT}, \cite{TWY}, 
\cite{TY} and \cite{Z}. Let us start by recalling the definitions and
possible comparisons of different types of mappings. In the sequel, we take
the mapping $T:C\rightarrow H$ unless otherwise specified. $T$\ is called
contractive if there exists a real number $\alpha $ with $0<\alpha <1$ such
that $\left\Vert Tx-Ty\right\Vert \leq \alpha \left\Vert x-y\right\Vert $
for all $x,y\in C$. $T$ is said to be nonexpansive if $\left\Vert
Tx-Ty\right\Vert \leq \left\Vert x-y\right\Vert $ for all $x,y\in C.$ $T$ is
said to be quasi-nonexpansive if for $p\in F(T),$ $\left\Vert
Tx-p\right\Vert \leq \left\Vert x-p\right\Vert $ for all $x\in C.$ $T\ $is
called quasi-contractive (due to Berinde \cite{B}) if there exist real
numbers $\alpha $ with $0<\alpha <1$\ and $L\geq 0$ such that $\left\Vert
Tx-Ty\right\Vert \leq \alpha \left\Vert x-y\right\Vert +L\left\Vert
x-Tx\right\Vert $ for all $x,y\in C$. Note that class of quasi-contractive
mappings already contains contractions, Kannan, Chatterji and Zamfiresu
operators (for definitions see \cite{B}). Takahashi et al. \cite{TWY}
introduced a broader class of nonlinear mappings which contains the class of
contractive mappings and the class of generalized hybrid mappings. $T$ is
called normally generalized hybrid if there exist $\alpha ,\beta ,\gamma
,\delta \in \mathbb{R}$ such that%
\begin{equation}
\alpha \left\Vert Tx-Ty\right\Vert ^{2}+\beta \left\Vert x-Ty\right\Vert
^{2}+\gamma \left\Vert Tx-y\right\Vert ^{2}+\delta \left\Vert x-y\right\Vert
^{2}\leq 0  \label{ngm}
\end{equation}%
for all $x,y\in C.$ A normally generalized hybrid mapping with a fixed point
is quasi-nonexpansive. Moreover, a normally generalized hybrid mapping with $%
\alpha =1,\beta =\gamma =0,-1<\delta <0,$ is a contractive mapping. However,
this class does not contain the class of quasi-contractive mappings due to
Berinde \cite{B}. Finally, we have also found another class of mappings in 
\cite{GT} which was originally introduced by Kawasaki and Takahashi \cite%
{KwT} and called "widely more generalized hybrid" in a Hilbert space. $T$ is
called "widely more generalized hybrid" if there exist $\alpha ,\beta
,\gamma ,\delta ,\varepsilon ,\varsigma ,\eta \in \mathbb{R}$ such that%
\begin{equation}
\left. 
\begin{array}{c}
\alpha \left\Vert Tx-Ty\right\Vert ^{2}+\beta \left\Vert x-Ty\right\Vert
^{2}+\gamma \left\Vert Tx-y\right\Vert ^{2}+\delta \left\Vert x-y\right\Vert
^{2} \\ 
+\varepsilon \left\Vert x-Tx\right\Vert ^{2}+\varsigma \left\Vert
y-Ty\right\Vert ^{2}+\eta \left\Vert (x-Tx)-(y-Ty)\right\Vert ^{2}\leq 0%
\end{array}%
\right.  \label{wmgm}
\end{equation}%
for all $x,y\in C.$ They noted that the class of widely more generalized
hybrid mapping contains the class of normally generalized hybrid mappings
but not of quasi-nonexpansive mappings generally even with having a fixed
point.

Our purpose in this paper is (i) to introduce the concept of further
generalized hybrid mappings (ii) to introduce the concept of common
attractive points (CAP) (iii) to write and use Picard-Mann iterative process
for two mappings. We approximate common attractive points of further
generalized hybrid mappings by using iterative process due to Khan \cite{SHK}
generalized to the case of two mappings in Hilbert spaces without closedness
of $C$. First we introduce further generalized hybrid mappings as another
generalization of normally generalized hybrid mappings. $T$ is called a
further generalized mapping if there exist $\alpha ,\beta ,\gamma ,\delta
,\epsilon \in \mathbb{R}$ such that%
\begin{equation}
\left. 
\begin{array}{c}
\alpha \left\Vert Tx-Ty\right\Vert ^{2}+\beta \left\Vert x-Ty\right\Vert
^{2}+\gamma \left\Vert Tx-y\right\Vert ^{2} \\ 
+\delta \left\Vert x-y\right\Vert ^{2}+\varepsilon \left\Vert
x-Tx\right\Vert ^{2}\leq 0%
\end{array}%
\right.  \label{sgm}
\end{equation}%
for all $x,y\in C.$ Obviously, this is a generalization of $(\ref{ngm})$
when $\varepsilon =0.$ It is noteworthy that it contains the class of
quasi-nonexpansive mappings, quasi-contractive mappings due to Berinde \cite%
{B} and in turm, contractive mappings, Kannan mappings, Chatterjea mappings,
Zamfirescu mappings. For definitions of these mappings, see for example, 
\cite{B}. To see that $(\ref{sgm})$ actually contains quasi-contractive
mappings due to Berinde \cite{B}, choose $\alpha =1,\beta =\gamma =0,\delta
\in (-1,0)$, $\varepsilon \in (-\infty ,0]$ and then using\emph{\ }$%
a^{2}+b^{2}\leq (a+b)^{2}$for all nonnegative real number\emph{s }$a,b.\ $%
Recall that quasi-contractive mappings due to Berinde \cite{B} are not
contained in $(\ref{ngm})$. Apparently, this seems a special case of "widely
more generalized hybrid" mappings $(\ref{wmgm})$ when $\varsigma =\eta =0.$
However,\ our class not only constitutes a simple generalization of $(\ref%
{ngm})$ but also as mentioned above contains the class of quasi-nonexpansive
mappings when it has a fixed point contrary to "widely more generalized
hybrid" mappings $(\ref{wmgm})$. So the results obtained by our new mapping
will not only be more general but also simpler.

Now, we introduce the concept of common attractive points for two mappings $%
S $ and $T$ denoted and defined as follows:%
\begin{equation*}
CAP(S,T)=\{z\in H:\max (\left\Vert Sx-z\right\Vert ,\left\Vert
Tx-z\right\Vert \leq \left\Vert x-z\right\Vert \}
\end{equation*}%
for all $x\in C.$ Obviously, $z\in CAP(S,T)$ means that $z\in A(S)$ as well
as $z\in A(T).$ Note also that: $CAP(S,T)=A(T)$ when $S=T$ .

Recall that Mann iterative process is:%
\begin{equation}
\begin{cases}
x_{1}=x\in C, \\ 
x_{n+1}=\left( 1-\alpha _{n}\right) x_{n}+\alpha _{n}Tx_{n},\;n\in 
%TCIMACRO{\U{2115} }%
%BeginExpansion
\mathbb{N}
%EndExpansion
.%
\end{cases}
\label{Mann}
\end{equation}%
Khan \cite{SHK} introduced a new iterative process called Picard-Mann hybrid
iterative process: 
\begin{equation}
\begin{cases}
x_{1}=x\in C, \\ 
x_{n+1}=Ty_{n}, \\ 
y_{n}=\left( 1-\alpha _{n}\right) x_{n}+\alpha _{n}Tx_{n},\;n\in 
%TCIMACRO{\U{2115}}%
%BeginExpansion
\mathbb{N}%
%EndExpansion
\end{cases}%
\;  \label{Safeer}
\end{equation}%
where $\{\alpha _{n}\}$ is in $(0,1).$ It was proved to be independent but
faster than all Picard, Mann, Ishikawa processes. Finally, we generalize it
to the case of two mappings $S$ and $T$ as follws:%
\begin{equation}
\begin{cases}
x_{1}=x\in C, \\ 
x_{n+1}=Sy_{n}, \\ 
y_{n}=\left( 1-\alpha _{n}\right) x_{n}+\alpha _{n}Tx_{n},\;n\in 
%TCIMACRO{\U{2115}}%
%BeginExpansion
\mathbb{N}%
%EndExpansion
\end{cases}%
\;  \label{S1}
\end{equation}%
where $\{\alpha _{n}\}$ is in $(0,1).$\ This process reduces to Mann if $S=I$%
, the identity mapping and at the same time deals with common attactive
points.

In short, we approximate common attractive points of $(\ref{sgm})$\ through $%
(\ref{S1})$ in Hilbert spaces without closedness of $C$. Our results are
generalizations and improvements of several results in the literature as
mentioned later in this paper.

\section{Main Results}

Let us first give some useful properties of $CAP(S,T)$ on the lines similar
to Lemmas $\ref{ATFT1}$, $\ref{CC}$ and $\ref{ATFT2}.$ For the sake of
simplicity, we take the parameters $\alpha ,\beta ,\gamma ,\delta
,\varepsilon \in \mathbb{R}$ same for the two further generalized hybrid
mappings $S,T$ as defined in $(\ref{sgm})$.

\begin{lemma}
\label{1}Let $H$ be a real Hilbert space and let $C$ be a nonempty closed
convex subset of $H$. Let $S,T:C\rightarrow C$ be two mappings. If $%
CAP(S,T)\neq \varnothing $ then $F(S)\cap F(T)\neq \varnothing .$ In
particular, if $z\in CAP(S,T)$, then $P_{C}z\in F(S)\cap F(T).$
\end{lemma}

\begin{proof}
Let $z\in CAP(S,T).$ Then $z\in A(S)$ and $z\in A(T)$ (and of course $z\in H$%
). Thus there is a unique element $u=P_{C}z\in C$ such that $\left\Vert
u-z\right\Vert \leq \left\Vert y-z\right\Vert $ for all $y\in C.$ Now $Tu\in
C$ implies $\left\Vert u-z\right\Vert \leq \left\Vert Tu-z\right\Vert .$ On
the other hand, $z\in A(T),$ therefore $\left\Vert Ty-z\right\Vert \leq
\left\Vert y-z\right\Vert $ for all $y\in C$ and in particular, $\left\Vert
Tu-z\right\Vert \leq \left\Vert u-z\right\Vert .$ Thus $\left\Vert
Tu-z\right\Vert \leq \left\Vert u-z\right\Vert \leq \left\Vert
Tu-z\right\Vert $ and hence $u\in F(T).$ Similarly, $u\in F(S)$ and so $%
F(S)\cap F(T)\neq \varnothing $ and $u=P_{C}z\in F(S)\cap F(T).$
\end{proof}

\begin{lemma}
\label{CAPCC}Let $H$ be a real Hilbert space and let $C$ be a nonempty
subset of $H$. Let $S,T:C\rightarrow C$ be two mappings. Then $CAP(S,T)$ is
a closed and convex subset of $H$.
\end{lemma}

\begin{proof}
Since intersection of two closed sets is closed and that of two convex sets
is convex, the proof follows on the lines similar to Lemma 2.3 of \cite{TT11}%
.
\end{proof}

\begin{lemma}
Let $H$ be a real Hilbert space and let $C$ be a nonempty subset of $H$. Let 
$S,T:C\rightarrow H$ be two quasi-nonexpansive mappings.Then $CAP(S,T)\cap
C=F(S)\cap F(T).$
\end{lemma}

\begin{proof}
Let $z\in CAP(S,T)\cap C.$ Then, by definition, $\max (\left\Vert
Sx-z\right\Vert ,$\newline
$\left\Vert Tx-z\right\Vert \leq \left\Vert x-z\right\Vert )$ for all $x\in
C.$ In particular, choosing $x=z\in C,$ we obtain $\max (\left\Vert
Sz-z\right\Vert ,\left\Vert Tz-z\right\Vert )\leq 0.$ That is, $z\in
F(S)\cap F(T).$ Conversely, since $z\in F(S)\cap F(T)$ and $S,T:C\rightarrow
H$ are quasi-nonexpansive mappings, we have $\left\Vert Sx-z\right\Vert \leq
\left\Vert x-z\right\Vert ,~\left\Vert Tx-z\right\Vert \leq \left\Vert
x-z\right\Vert \ $for all $x\in C.$ This implies that $\max (\left\Vert
Sx-z\right\Vert ,\left\Vert Tx-z\right\Vert )\leq \left\Vert x-z\right\Vert
\ $for all $x\in C.$ Clearly $z\in C.$ Hence $z\in CAP(S,T)\cap C$. This
completes the proof.
\end{proof}

Our next result is an existence theorem on common attractive points of two
further generalized hybrid mappings $(\ref{sgm})$ without any use of
closedness and convexity.\ This result is followed by some important remarks
on comparing it with some results in the current literature.

\begin{theorem}
\label{Khan1}Let $H$ be a real Hilbert space and let $C$ be a nonempty
subset of $H$. Let $S,T:C\rightarrow C$ be two further generalized hybrid
mappings as defined in $(\ref{sgm})$ which satisfy $\alpha +\beta +\gamma
+\delta \geq 0,\varepsilon \geq 0$ and either $\alpha +\beta >0$ or $\alpha
+\gamma >0.$ Then $CAP(S,T)\neq \varnothing $ if and only if there exists $%
z\in C$ such that $\{S^{n}z\cap T^{n}z,n=0,1,2,...\}$ is bounded.

\begin{proof}
Suppose that $CAP(S,T)\neq \varnothing $\ and $z\in CAP(S,T).$ Then\ by
definition, $\max (\left\Vert Sx-z\right\Vert ,\left\Vert Tx-z\right\Vert
)\leq \left\Vert x-z\right\Vert \ $for all $x\in C.$\ This means that $%
\left\Vert S^{n+1}x-z\right\Vert \leq \left\Vert S^{n}x-z\right\Vert $ and $%
\left\Vert T^{n+1}x-z\right\Vert )\leq \left\Vert T^{n}x-z\right\Vert \ $for
all $x\in C$ and $n\in \mathbb{N}.$\ That is, $\{S^{n}z\cap
T^{n}z,n=0,1,2,...\}$ is bounded.

Conversely, suppose that there exists $z\in C$ such that $\{S^{n}z\cap
T^{n}z,n=0,1,2,...\}$ is bounded. This gives that there exists $z\in C$ such
that $\{S^{n}z,n=0,1,2,...\}$ is bounded as well as $\{T^{n}z,n=0,1,2,...\}$
is bounded.$\ $Suppose that $\max (\left\Vert Sx-z\right\Vert ,\left\Vert
Tx-z\right\Vert )=\left\Vert Tx-z\right\Vert .$ After doing long
calculations on the lines similar to\ Theorem 8 of \cite{GT}, we find that
there exists a $p\in H$ such that $\left\Vert Tx-p\right\Vert ^{2}\leq
\left\Vert x-p\right\Vert ^{2}.$ This means that $p\in A(T).$ But by our
supposition on maximum, we get $\left\Vert Sx-p\right\Vert ^{2}\leq
\left\Vert x-p\right\Vert ^{2}.$ Thus $CAP(S,T)\neq \varnothing .$ In case, $%
\max (\left\Vert Sx-z\right\Vert ,$\newline
$\left\Vert Tx-z\right\Vert )=\left\Vert Sx-z\right\Vert ,$ we can get the
result by interchanging the role of $S$ and $T.$
\end{proof}
\end{theorem}

This theorem consitutes a generalization of Theorem 3.1 of \cite{TWY} and
the results generalized therein when $S=T$ and $\varepsilon =0.$ Clearly
this theorem handles existence of common attractive points, so it is
independent of Theorem 8 of \cite{GT}. But a special case of our result when 
$S=T$ can be obtained from Theorem 8 of \cite{GT} by choosing $\varsigma
=\eta =0$.\ However, even in this special case, it is more general in the
sense that our class of mappings is simpler and always covers the class of
quasi-nonexpansive mappings as opposed to Theorem 8 of \cite{GT}. The same
holds for all the results of \cite{GT}.

Let us now come to one of our main targets of proving a weak convergence
theorem in Hilber spaces without needing closedness of $C$.

\begin{theorem}
\label{Main}Let $H$ be a real Hilbert space and let $C$ be a nonempty convex
subset of $H$. Let $S,T:C\rightarrow C$ be two further generalized hybrid
mappings as defined in $(\ref{sgm})$ which satisfy $\alpha +\beta +\gamma
+\delta \geq 0,\varepsilon \geq 0$ and either $\alpha +\beta >0$ or $\alpha
+\gamma >0.$ Let $CAP(S,T)\neq \varnothing .$ If $\{x_{n}\}$ is defined by $(%
\ref{S1}),$where $\{\alpha _{n}\}$ is a sequence in $(0,1)$ with $\liminf
\alpha _{n}(1-\alpha _{n})>0,$ then $\{x_{n}\}$ converges weakly to a point $%
q\in CAP(S,T).$ Moreover, $q=\lim_{n\rightarrow \infty }Px_{n\text{ }}$where 
$P$ is projection of $H$ onto $CAP(S,T).$
\end{theorem}

\begin{proof}
Let $z\in CAP(S,T).$ Then by $(\ref{S1}),$%
\begin{eqnarray*}
\left\Vert y_{n}-z\right\Vert ^{2} &=&\left\Vert \left( 1-\alpha _{n}\right)
x_{n}+\alpha _{n}Tx_{n}-z\right\Vert ^{2} \\
&\leq &\left( 1-\alpha _{n}\right) \left\Vert x_{n}-z\right\Vert ^{2}+\alpha
_{n}\left\Vert Tx_{n}-z\right\Vert ^{2} \\
&\leq &\left( 1-\alpha _{n}\right) \left\Vert x_{n}-z\right\Vert ^{2}+\alpha
_{n}\left\Vert x_{n}-z\right\Vert ^{2} \\
&=&\left\Vert x_{n}-z\right\Vert ^{2}
\end{eqnarray*}%
and so 
\begin{eqnarray*}
\left\Vert x_{n+1}-z\right\Vert ^{2} &=&\left\Vert Sy_{n}-z\right\Vert ^{2}
\\
&\leq &\left\Vert y_{n}-z\right\Vert ^{2} \\
&\leq &\left\Vert x_{n}-z\right\Vert ^{2}.
\end{eqnarray*}%
Thus 
\begin{equation}
\left\Vert x_{n+1}-z\right\Vert ^{2}\leq \left\Vert x_{n}-z\right\Vert ^{2}
\label{increasing}
\end{equation}%
for all $n\in \mathbb{N}$. Thus $\lim_{n\rightarrow \infty }\left\Vert
x_{n}-z\right\Vert ^{2}$ exists and so $\{x_{n}\}$ must be bounded.

\noindent Since $H$ is a Hilbert space, so 
\begin{eqnarray*}
\left\Vert x_{n+1}-z\right\Vert ^{2} &=&\left\Vert Sy_{n}-z\right\Vert ^{2}
\\
&\leq &\left\Vert y_{n}-z\right\Vert ^{2} \\
&=&\left\Vert \left( 1-\alpha _{n}\right) x_{n}+\alpha
_{n}Tx_{n}-z\right\Vert ^{2} \\
&=&\left. 
\begin{array}{c}
\left( 1-\alpha _{n}\right) \left\Vert x_{n}-z\right\Vert ^{2}+\alpha
_{n}\left\Vert Tx_{n}-z\right\Vert ^{2} \\ 
-\alpha _{n}\left( 1-\alpha _{n}\right) \left\Vert Tx_{n}-x_{n}\right\Vert
^{2}%
\end{array}%
\right. \\
&\leq &\left. 
\begin{array}{c}
\left( 1-\alpha _{n}\right) \left\Vert x_{n}-z\right\Vert ^{2}+\alpha
_{n}\left\Vert x_{n}-z\right\Vert ^{2} \\ 
-\alpha _{n}\left( 1-\alpha _{n}\right) \left\Vert Tx_{n}-x_{n}\right\Vert
^{2}%
\end{array}%
\right. \\
&=&\left\Vert x_{n}-z\right\Vert ^{2}-\alpha _{n}\left( 1-\alpha _{n}\right)
\left\Vert Tx_{n}-x_{n}\right\Vert ^{2}.
\end{eqnarray*}%
This implies that%
\begin{equation*}
\alpha _{n}\left( 1-\alpha _{n}\right) \left\Vert Tx_{n}-x_{n}\right\Vert
^{2}\leq \left\Vert x_{n}-z\right\Vert ^{2}-\left\Vert x_{n+1}-z\right\Vert
^{2}.
\end{equation*}%
Now using the condition $\liminf \alpha _{n}(1-\alpha _{n})>0$ and the above
proved fact that $\lim_{n\rightarrow \infty }\left\Vert x_{n}-z\right\Vert
^{2}$ exists, we have 
\begin{equation*}
\lim_{n\rightarrow \infty }\left\Vert Tx_{n}-x_{n}\right\Vert =0.
\end{equation*}

\noindent We have also proved in the above lines that $\{x_{n}\}$ is a
bounded sequence, therefore we have its subsequence $\{x_{n_{j}}\}$ such
that $x_{n_{j}}\rightharpoonup q\in C.$ Since $T:C\rightarrow C$ is a
further generalized mapping, therefore%
\begin{equation*}
\left. 
\begin{array}{c}
\alpha \left\Vert Tx_{n_{j}}-Ty\right\Vert ^{2}+\beta \left\Vert
x_{n_{j}}-Ty\right\Vert ^{2}+\gamma \left\Vert Tx_{n_{j}}-y\right\Vert ^{2}
\\ 
+\delta \left\Vert x_{n_{j}}-y\right\Vert ^{2}+\varepsilon \left\Vert
x_{n_{j}}-Tx_{n_{j}}\right\Vert ^{2}\leq 0%
\end{array}%
\right.
\end{equation*}%
and so%
\begin{equation*}
\left. 
\begin{array}{c}
\alpha (\left\Vert Tx_{n_{j}}-x_{n_{j}}\right\Vert ^{2}+\left\Vert
x_{n_{j}}-Ty\right\Vert ^{2}+2\left\langle
Tx_{n_{j}}-x_{n_{j}},x_{n_{j}}-Ty\right\rangle ) \\ 
+\beta \left\Vert x_{n_{j}}-Ty\right\Vert ^{2}+\gamma \left\Vert
Tx_{n_{j}}-y\right\Vert ^{2} \\ 
+\delta \left\Vert x_{n_{j}}-y\right\Vert ^{2}+\varepsilon \left\Vert
x_{n_{j}}-Tx_{n_{j}}\right\Vert ^{2}\leq 0.%
\end{array}%
\right.
\end{equation*}

\noindent Making use of \ Banach limit $\mu ,$ we get%
\begin{equation*}
(\alpha +\beta )\mu _{n}\left\Vert x_{n_{j}}-Ty\right\Vert ^{2}+(\gamma
+\delta )\mu _{n}\left\Vert x_{n_{j}}-y\right\Vert ^{2}\leq 0.
\end{equation*}

\noindent This yields that%
\begin{equation*}
\left. 
\begin{array}{c}
(\alpha +\beta )\mu _{n}[\left\Vert x_{n_{j}}-y\right\Vert ^{2}+\left\Vert
y-Ty\right\Vert ^{2}+2\left\langle x_{n_{j}}-y,y-Ty\right\rangle ] \\ 
+(\gamma +\delta )\mu _{n}\left\Vert x_{n_{j}}-y\right\Vert ^{2}\leq 0.%
\end{array}%
\right.
\end{equation*}

\noindent Thus%
\begin{equation*}
\left. 
\begin{array}{c}
(\alpha +\beta +\gamma +\delta )\mu _{n}[\left\Vert x_{n_{j}}-y\right\Vert
^{2} \\ 
+(\alpha +\beta )\left\Vert y-Ty\right\Vert ^{2}+2(\alpha +\beta )\mu
_{n}\left\langle x_{n_{j}}-y,y-Ty\right\rangle \leq 0.%
\end{array}%
\right.
\end{equation*}%
\noindent But $\alpha +\beta +\gamma +\delta \geq 0$, therefore%
\begin{equation*}
(\alpha +\beta )\left\Vert y-Ty\right\Vert ^{2}+2(\alpha +\beta )\mu
_{n}\left\langle x_{n_{j}}-y,y-Ty\right\rangle \leq 0.
\end{equation*}%
\noindent Since $x_{n_{j}}\rightharpoonup q,$ therefore%
\begin{equation*}
(\alpha +\beta )\left\Vert y-Ty\right\Vert ^{2}+2(\alpha +\beta
)\left\langle q-y,y-Ty\right\rangle \leq 0.
\end{equation*}

\noindent Since $H$ is a Hilbert space so using 
\begin{equation}
2\left\langle u-v,p-w\right\rangle =\left\Vert u-w\right\Vert
^{2}+\left\Vert v-p\right\Vert ^{2}-\left\Vert u-p\right\Vert
^{2}-\left\Vert v-w\right\Vert ^{2}  \label{Useful}
\end{equation}%
in the above inequality, we have%
\begin{equation*}
\left. 
\begin{array}{c}
(\alpha +\beta )\left\Vert y-Ty\right\Vert ^{2} \\ 
+(\alpha +\beta )[\left\Vert q-Ty\right\Vert ^{2}+\left\Vert y-y\right\Vert
^{2}-\left\Vert q-y\right\Vert ^{2}-\left\Vert y-Ty\right\Vert ^{2}]\leq 0.%
\end{array}%
\right.
\end{equation*}

\noindent This implies that $(\alpha +\beta )[\left\Vert q-Ty\right\Vert
^{2}-\left\Vert q-y\right\Vert ^{2}]\leq 0.$ Since $(\alpha +\beta
)>0,\left\Vert q-Ty\right\Vert ^{2}-\left\Vert q-y\right\Vert ^{2}\leq 0$.
Similarly, we get $\left\Vert q-Sy\right\Vert ^{2}-\left\Vert q-y\right\Vert
^{2}\leq 0$ and hence $q\in CAP(S,T).$ Next we prove that $%
x_{n}\rightharpoonup q$ by proving that any two subsequences of $\{x_{n}\}$
converge to the same limit $q.$ Let $x_{n_{j}}\rightharpoonup q_{1\text{ }}$%
and $x_{n_{k}}\rightharpoonup q_{2.}$ By what we have just proved, $q_{1%
\text{ }}$and $q_{2}$ belong to $CAP(S,T)$ and from the initial steps of
this proof we conclude that $\lim_{n\rightarrow \infty }(\left\Vert
x_{n}-q_{1}\right\Vert ^{2}-\left\Vert x_{n}-q_{2}\right\Vert ^{2})$ exists,
call it $\ell $. Now using $(\ref{Useful})$ again, $2\left\langle
x_{n},q_{2}-q_{1}\right\rangle =\left\Vert x_{n}-q_{1}\right\Vert
^{2}+\left\Vert q_{2}\right\Vert ^{2}-\left\Vert x_{n}-q_{2}\right\Vert
^{2}-\left\Vert q_{1}\right\Vert ^{2}.$ This yields $\left\Vert
x_{n}-q_{1}\right\Vert ^{2}-\left\Vert x_{n}-q_{2}\right\Vert
^{2}=2\left\langle x_{n},q_{2}-q_{1}\right\rangle -\left\Vert
q_{2}\right\Vert ^{2}+\left\Vert q_{1}\right\Vert ^{2}.$ Thus

$\left\Vert x_{nj}-q_{1}\right\Vert ^{2}-\left\Vert
x_{n_{j}}-q_{2}\right\Vert ^{2}=2\left\langle
x_{n_{j}},q_{2}-q_{1}\right\rangle -\left\Vert q_{2}\right\Vert
^{2}+\left\Vert q_{1}\right\Vert ^{2}$ and

$\left\Vert x_{n_{k}}-q_{1}\right\Vert ^{2}-\left\Vert
x_{n_{k}}-q_{2}\right\Vert ^{2}=2\left\langle
x_{n_{k}},q_{2}-q_{1}\right\rangle -\left\Vert q_{2}\right\Vert
^{2}+\left\Vert q_{1}\right\Vert ^{2}.$

Now taking weak limit on the above two equations and making use of $%
x_{n_{j}}\rightharpoonup q_{1\text{ }}$and $x_{n_{k}}\rightharpoonup q$, we
get

$\ell =2\left\langle q_{1},q_{2}-q_{1}\right\rangle -\left\Vert
q_{2}\right\Vert ^{2}+\left\Vert q_{1}\right\Vert ^{2}$

$\ell =2\left\langle q_{2},q_{2}-q_{1}\right\rangle -\left\Vert
q_{2}\right\Vert ^{2}+\left\Vert q_{1}\right\Vert ^{2}.$

\noindent Subtracting we get, $2\left\langle
q_{1}-q_{2},q_{2}-q_{1}\right\rangle =0$ and hence $q_{1}=q_{2}.$ In turn, $%
x_{n}\rightharpoonup q\in CAP(S,T).$

Finally, we show that $q=\lim_{n\rightarrow \infty }Px_{n}$ where $P$\ is
the projection of $H$ onto $CAP(S,T).$ Now from (\ref{increasing}), $%
\left\Vert x_{n+1}-z\right\Vert \leq \left\Vert x_{n}-z\right\Vert $ for all 
$z\in CAP(S,T)$ and $n\in \mathbb{N}.$ Since $CAP(S,T)$ is closed and convex
by Lemma \ref{CAPCC}, applying Lemma \ref{TT}, $\lim_{n\rightarrow \infty
}Px_{n}=p$ for some $p\in CAP(S,T).$ It is well known for projections that $%
\left\langle x_{n}-Px_{n},Px_{n}-z\right\rangle \geq 0$ for all $z\in
CAP(S,T)$ and $n\in \mathbb{N}.$ Therefore, $\left\langle
q-p,p-z\right\rangle \geq 0$ for all $z\in CAP(S,T)$ and in particular, $%
\left\langle q-p,p-q\right\rangle \geq 0.$ Hence, $q=p=\lim_{n\rightarrow
\infty }Px_{n}.$
\end{proof}

Although the following is a corollary to the above theorem, yet it is a new
result in itself. As already mentioned, the iterative process $(\ref{Safeer}%
) $ is independent but faster than several iterative processes, therefore\
this corollary has its own standing.

\begin{corollary}
\label{T}Let$\ H,C,$ $T$ and $\alpha ,\beta ,\gamma ,\delta ,\varepsilon $
be as in Theorem $\ref{Main}.$ If $\{x_{n}\}$ is defined by the itertive
process $(\ref{Safeer}),$ where $\{\alpha _{n}\}$ is a sequence in $(0,1)$
with $\liminf \alpha _{n}(1-\alpha _{n})>0,$ then $\{x_{n}\}$ converges
weakly to a point $q\in A(T).$ Moreover, $q=\lim_{n\rightarrow \infty }Px_{n%
\text{ }}$where $P$\ is the projection of $H$ onto $A(T).$
\end{corollary}

\begin{proof}
Choose $S=T$ in the above theorem.
\end{proof}

\begin{corollary}
\label{I}Let $\ H,C,\ T$ and $\alpha ,\beta ,\gamma ,\delta ,\varepsilon $
be as in Theorem $\ref{Main}.$ If $\{x_{n}\}$ is defined by Mann itertive
process $(\ref{Mann}),$where $\{\alpha _{n}\}$ is a sequence in $(0,1)$ with 
$\liminf \alpha _{n}(1-\alpha _{n})>0,$ then $\{x_{n}\}$ converges weakly to
a point $q\in A(T).$ Moreover, $q=\lim_{n\rightarrow \infty }Px_{n\text{ }}$%
where $P$\ is the projection of $H$ onto $A(T).$
\end{corollary}

\begin{proof}
Choose $S=I$ in the above theorem.
\end{proof}

Now we give some remarks on how our above results are generalizations and
improvements of the results in the existing literature.

\begin{remark}
\begin{enumerate}
\item Theorem 5.1 of \cite{TWY} can now be obtained by choosing either $%
S=I,\varepsilon =0$ in Theorem $\ref{Main}$ or $\varepsilon =0$ in Corollary 
$\ref{I}.$

\item Corollary $\ref{I}$ can be viewed as an improvement and extension of
Theorem 8 of \cite{GT} in the sense that (i)our class of mappings is simpler
(ii)it contains the class of quasi nonexpansive mappings as opposed to \cite%
{GT}. Corollary $\ref{T}$ not only keeps this sense but also gives faster
convergence (see \cite{SHK}).

\item Corollary $\ref{T}$ (leave alone our Theorem $\ref{Main}$) generalizes
Corollary $4.3$ of Zheng \cite{Z} in two ways: We do not need closedness of $%
C$ and the class of our mappings is much more general than that of \cite{Z}.

\item Of course, all corresponding results generalized in \cite{TWY} and 
\cite{GT} are part and parcel of the above remarks.
\end{enumerate}
\end{remark}

If, in addition, we use the closedness of $C$ in Theorem $\ref{Main}$, then
we have the following:

\begin{theorem}
Let $H$ be a real Hilbert space and let $C$ be a nonempty closed convex
subset of $H$. Let $S,T:C\rightarrow C$ be two further generalized hybrid
mappings as defined in $(\ref{sgm})$ which satisfy $\alpha +\beta +\gamma
+\delta \geq 0,\varepsilon \geq 0$ and either $\alpha +\beta >0$ or $\alpha
+\gamma >0.$ Let $CAP(S,T)\neq \varnothing .$ If $\{x_{n}\}$ is defined by $(%
\ref{S1}),$ where $\{\alpha _{n}\}$ is a sequence in $(0,1)$ with $\liminf
\alpha _{n}(1-\alpha _{n})>0,$ then $\{x_{n}\}$ converges weakly to a point $%
P_{C}q\in F(S)\cap F(T)$ where $P_{C}:H\rightarrow C$ the metric projection
\end{theorem}

\begin{proof}
By Theorem $\ref{Main},q\in CAP(S,T).$ Now using Lemma $\ref{1},P_{C}q\in
F(S)\cap F(T)$ as desired.
\end{proof}

\noindent \textbf{Competing interests.} \textquotedblleft The author(s)
declare(s) that they have no competing interests\textquotedblright .

\begin{acknowledgement}
I remain grateful to my PhD advisor Professor\ Wataru Takahashi for what I\
have learnt from him throughout my career.
\end{acknowledgement}

\end{document}